\documentclass{article}
\usepackage[dvipsnames]{xcolor}
\usepackage{graphicx} 
\usepackage{amsmath, amsthm, amssymb}
\usepackage{thmtools}
\usepackage{hyperref,cleveref}
\usepackage{tcolorbox}
\usepackage{todonotes}
\usepackage[ruled]{algorithm2e}
\usepackage{tikz}
\usepackage{enumitem}

\theoremstyle{plain}
\newtheorem{thm}{Theorem}[section]
\newtheorem{lem}[thm]{Lemma}
\newtheorem{prop}[thm]{Proposition}

\theoremstyle{definition}

\newtheorem{ex}[thm]{Example}
\newtheorem{obs}[thm]{Observation}
\newtheorem{question}[thm]{Problem}

\theoremstyle{remark}
\newtheorem{remark}[thm]{Remark}

\newcommand{\RR}{\mathbb{R}}
\newcommand{\ZZ}{\mathbb{Z}}

\newcommand{\bm}[1]{\boldsymbol{#1}}

\DeclareMathOperator{\cone}{cone}

\DeclareMathOperator{\col}{col}
\DeclareMathOperator{\row}{row}
\DeclareMathOperator{\cols}{cols}
\DeclareMathOperator{\rank}{rank}
\DeclareMathOperator{\rankzn}{rank_{\ZZ}^+}
\DeclareMathOperator{\rankp}{rank^{+}}

\title{Matrices of nonnegative integer rank two}
\author{Jo\~ao Gouveia, Amy Wiebe}
\date{\today}

\begin{document}

\maketitle


\begin{abstract}
    
The nonnegative integer rank of a matrix is a variant of the classical nonnegative rank, introduced in the 1980s, where factorizations are required to have integer entries. While computing nonnegative integer rank is generally very hard, we focus on a fundamental special case: determining when a rank 2 nonnegative integer matrix has nonnegative integer rank equal to 2 (the ``rank2 problem"). Although this problem is trivial in the continuous case, in this context it is surprisingly rich.

We provide a geometric reformulation in terms of affine semigroups and rational cones in the plane, which yields new structural insights. We show that any rank 2 integer matrix can be reduced to a $3\times 3$ matrix which has nonnegative integer rank $2$
if and only if the original one also has nonnegative integer rank $2$, with the reduction computable in polynomial time. This reduction reveals that the difficulty of the rank2 problem is already captured by small matrices.
Building on this geometric framework, we also develop an algorithm that solves the rank2 problem by strategically searching for integer generators within bounded regions of the associated cone. Although the theoretical worst-case complexity remains high, numerical experiments demonstrate that the algorithm performs efficiently in practice. 
\end{abstract}

\section{Introduction}

The nonnegative rank of a matrix,  $\rankp(A)$, has been studied since the 1970s (see \cite{berman1973rank}) and has developed a large body of literature in recent years due to its connection with extension complexities of polytopes and nonnegative matrix factorizations. A modern presentation of the topic and its applications can be found in \cite{gillis2020nonnegative}. 
In this paper, we will study a lesser known integer variant of this quantity, the \emph{nonnegative integer rank} of a matrix $A$, denoted by $\rankzn(A)$, which was introduced in the 1980s (see \cite{gregory1983semiring}). 

Both these ranks are special cases of what is known as a semiring or a factorization rank (see \cite{gregory1983semiring}). Given $A \in \ZZ_+^{n \times m}$, for any semiring $S$ that contains $\ZZ_+$ one can define $\rank_S(A)$ as the smallest $k$ for which there exist $B \in S^{n \times k}$ and $C \in S^{k \times m}$ such that $A=BC$. 

The \emph{nonnegative integer rank} of matrix $A$, denoted by $\rankzn(A)$, is the special case when $S=\ZZ_+$. For $S=\RR_+$ we get the usual nonnegative rank $\rankp(A)$, and for $S=\RR$ we simply obtain the rank. In particular we have
for any matrix $A \in \ZZ_+^{n \times m}$, 
\begin{equation} \label{eq:basicineq}
    \rank(A) \leq \rankp(A) \leq \rankzn(A).
\end{equation}
See \cite{beasley1995rank} for a more general discussion of these ranks.

Although much is known about $\rankp(A)$, very few results exist for $\rankzn(A)$. For the special case where $A$ is a $\{0,1\}$-matrix, the nonnegative integer rank can be seen as the bipartite clique partition number of a bipartite graph. Under this guise it has been extensively studied,  both in the computer science and in the linear algebra community, since the groundbreaking work of Orlin \cite{orlin1977contentment} which proposed its NP-completeness as an open problem (later confirmed in \cite{jiang1993minimal}). In the more general framework, however, most of the research focuses on identifying classes of matrices for which certain ranks coincide (for example, see \cite{BEASLEY198833,beasley1995rank,seok2006factor}). There do not seem to exist lower bounds that are not directly obtained through \eqref{eq:basicineq}. 

Recently, the authors of this paper have shown that, similarly to the nonnegative rank, the nonnegative integer rank can be interpreted in terms of certain extension complexities \cite{integerslack24}. Additionally, in \cite{dong2018integer}, Dong et al study the closely related integer nonnegative matrix approximation problem, providing a specialized algorithm for it and making a good case that there are application advantages of working over the integers directly, while providing good applied examples where the nonnegative integer rank admit natural interpretations. These connections motivate us to develop the theoretical foundations of this notion, a program we set in motion in the present paper.

Computing the nonnegative integer rank of a matrix is a very hard problem, so we will begin to look into simpler versions of that problem. Given the basic inequalities presented in \eqref{eq:basicineq}, a natural question is the following.

\begin{question}[The minrank problem]\label{Q:minrank}
    Given a nonnegative matrix $A$ how do we check if $\rankzn(A)=\rank(A)$?
\end{question}

This question is still very hard. The corresponding question for the nonnegative rank (called ExactNMF) is known to be NP-hard \cite{vavasis2010complexity}, but when restricted to matrices of fixed rank $1,2$ or $3$, it becomes solvable in polynomial time \cite{gillis2012geometric}. In the case of rank one, the minrank problem is trivial, since every nonnegative integer matrix with rank one can easily be seen to have nonnegative integer rank one. However, contrary to what happens with the usual nonnegative rank, the integer rank two case  is not trivial.

For nonnegative matrices, having rank two implies having nonnegative rank two, but it is not true that it implies having nonnegative integer rank two. A known example of this is the matrix
 $$B=\begin{pmatrix}
 2 & 0 & 3 \\
 1 & 1 & 4 \\
 1 & 3 & 9
 \end{pmatrix}$$
from  \cite[Example 4.2]{beasley1995rank}, which has rank (and consequently nonnegative rank) two but nonnegative integer rank three. Moreover, in \cite{integerslack24} it is shown that the nonnegative integer rank can in fact be arbitrarily high for rank two matrices.  This means the following restricted version of the minrank problem is already of significant interest.
\begin{question}[The rank$2$ problem]
    Given a nonnegative matrix $A$ how do we check if $\rankzn(A)=2$?
    \label{Q:rank2}
\end{question}

On a different but related direction we have the result of \cite{LaffeySmigoc} where it shown for $2 \times 2$ integer completely positive matrices, that they must have an integer cp-factorization.

In this paper, we present an answer to \Cref{Q:rank2} in the form of an algorithm that is inspired by a geometric reformulation of the problem.
The remainder of this paper is organized as follows. In \Cref{sec:geometry}, we reinterpret the nonnegative integer rank of a rank 2 matrix as a geometric question about finding integer generators in a rational cone in $\RR^2$. We introduce a canonical diagram representation and show that for any rank 2 nonnegative integer matrix we can reduce it to a $3\times 3$ matrix which has nonnegative integer rank 2 if and only if the original one also has nonnegative integer rank 2, with the reduction procedure being computationally efficient. In \Cref{sec:alg}, we present our algorithm for solving the rank2 problem, which exploits the geometric structure by decomposing the planar cone introduced in \Cref{sec:geometry} and searching over a bounded region of candidate generators. We provide detailed numerical tests demonstrating that despite the theoretical hardness of the problem, our algorithm runs in reasonable time even for large matrices.

\section{A geometric view of the rank2 problem}\label{sec:geometry}



\subsection{A geometric reformulation} \label{subsec:geometric}

We start by noting that for the minrank problem, Problem \ref{Q:minrank}, if we have an $n \times m$ nonnegative integer matrix $A$ of rank $k$, a positive answer corresponds to the existence of nonnegative integer matrices $B$ and $C$ of sizes $n \times k$ and $m \times k$, respectively, such that $A=BC^\top$. This would imply that the column space of $A$ equals the column space of $B$, since every column of $A$ is a combination of the columns of $B$ and the dimension of $\col(B)$ is at most $k$.

If we consider the affine semigroup $\Gamma=\col(A) \cap \ZZ^n_+$ of dimension $k$, then the minrank problem can be framed as asking if there exist $k$ elements in $\Gamma$ that span a semigroup that contains all columns of $A$. This leads to the following equivalent version of the minrank problem. 

\begin{question}
Given a $k$-dimensional rational linear space $L$ in $\RR^n$ and a finite set $S$ of $m$ elements in $L \cap \ZZ_+^n$ that span $L$, can we find a set $B$ of $k$ elements in $L \cap \ZZ_+^n$ so that every element of $S$ is a nonnegative combination of elements of $B$?
\end{question}

Although this interpretation holds for every rank $k$, in the case of $k=2$ the geometry is particularly simple. Given a nonnegative matrix $A\in\ZZ^{n\times m}$ of rank 2, we define the cone $K_A = \col(A)\cap\RR_+^n$. Since $A$ has rank 2, $\col(A)$ is a 2-dimensional subspace, and hence we can find a transformation that maps the lattice $\col(A)\cap\ZZ^n$ bijectively to $\ZZ^2$ and $K_A$ into $\RR^2$. We can therefore think of the rank$2$ problem in the following way.

\begin{question} \label{q:cone}
Given rational cone $K$ in $\RR^2$ and $m$ integer points $a_1,...,a_m$ in it, are there two integer points $b_1$, $b_2$ in $K$ such that every $a_i$ is a nonnegative integer combination of $b_1$ and $b_2$?
\end{question}

\begin{ex} \label{ex:graphicexample}
Consider again the  rank $2$ matrix 
$$B=\begin{pmatrix}
 2 & 0 & 3 \\
 1 & 1 & 4 \\
 1 & 3 & 9
 \end{pmatrix}.$$
The lattice $\col(B)\cap\ZZ^3$ is generated by the vectors $(0, 1, 3)^\top$ and $(1, 0, -1)^\top$. The linear transformation that sends these points to $(1,0)$ and $(0,1)$, respectively, sends the columns of $B$ to $\{(1,2),(1,0),(4,3)\}$. The cone $\col(B)\cap\RR_+^n$, which is spanned by $(0,1,3)^\top$ and $(3,1,0)^\top$, is sent to the cone spanned by $(1,0)$ and $(1,3)$.
\begin{figure}
    \centering    \includegraphics[width=0.45\linewidth]{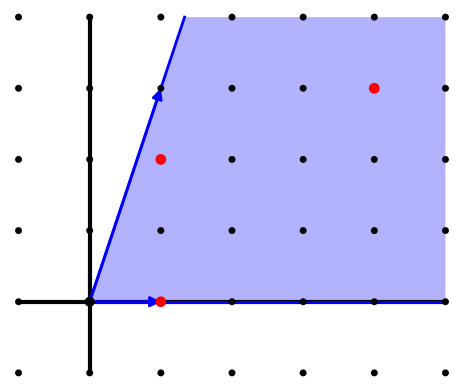}
    \caption{Graphical interpretation of Example \ref{ex:graphicexample}.}
    \label{fig:graphicexample}
\end{figure}

In Figure \ref{fig:graphicexample} we can see the graphical interpretation. We are asking if there are two lattice points in the blue region that generate all three red points. One can easily see that both $(1,0)$ and $(1,2)$ have to be there since no other points generate them, in other words, they are part of the Hilbert basis of the affine semigroup obtained by intersecting the cone with the integer lattice. Since they do not generate $(4,3)$ this immediately tells us the nonnegative integer rank of $B$ is indeed $3$.
 
\end{ex}

\begin{ex} \label{ex:graphicexample2}
Consider now the family of rank $2$ matrices 
$$B_t=\begin{pmatrix}
 t+1 & t & t-1 \\
 t & t & t \\
 t-1 & t & t+1
 \end{pmatrix},$$
 for any integer $t\geq 1$. The lattice $\col(B)\cap\ZZ^3$ is generated by the vectors $(1, 1, 1)^\top$ and $(1, 0, -1)^\top$. Applying the linear transformation that sends these points to $(0,1)$ and $(1,0)$, we get that the columns of $B_t$ are sent to $(-1,t)$ $(0,t)$ and $(1,t)$, and the cone $\col(B_t)\cap\RR_+^3$ is sent to the cone $K$ spanned by $(-1,1)$ and $(1,1)$. We can see the diagram for $t=4$ in Figure \ref{fig:graphicexample2}. For other values of $t$, the diagram will be the same but with the red points at height $t$.
\begin{figure}
    \centering    \includegraphics[width=0.5\linewidth]{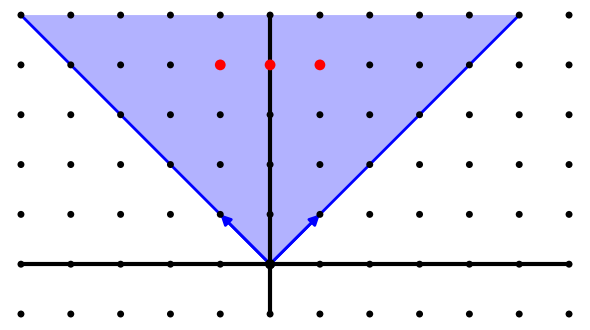}
    \caption{Graphical interpretation of Example \ref{ex:graphicexample2}.}
    \label{fig:graphicexample2}
\end{figure}

This case is a bit less obvious, as there are many possible generators to check (in theory, every point in the blue cone of height at most $t$). However, we will show that no choice of pairs of points can generate the three red points; thus, this family of matrices always has a nonnegative integer rank $3$.

Suppose that there exist $a$ and $b$ in $K \cap \ZZ^2$ that span all three points. Since the conic hull of $a$ and $b$ contains the three points, one must be to the `left' of $(-1,t)$ and the other to the `right' of $(1,t)$. Without loss of generality, we can then assume $a$ and $b$ to have negative and positive first coordinates, respectively. Then we know $$(0,t)=ka+(l+s)b, \ \ (-1,t)=(k+r)a+lb$$ for some nonnegative integers $k,l,r,s$. Subtracting we have $-ra+sb=(1,0)$, but both $-a$ and $b$ have the first coordinate at least $1$ so we would need $r=0$ and $s=1$ or $r=1$ and $s=0$, which would mean $b=(1,0)$ or $a=(-1,0)$ neither of which belongs to $K$.

\end{ex}

Note that the graphical diagram associated to a particular rank2 problem is not unique. This is because we can compose the isomorphism from $\col(A) \cap \ZZ^n$ to $\ZZ^2$, which we use to generate the diagram, with any linear automorphism of $\ZZ^2$. For consistency, we can define a canonical diagram corresponding to a matrix as follows.

 Consider one of the two minimal generators of the cone $K_A=\col(A)\cap\RR_+^n$. Let $(a,b)$ be its image in $\ZZ^2$. Minimality implies that $\gcd(a,b)=1$ so there exist $\alpha,\beta \in \ZZ$
so that $\alpha a + \beta b =1$. Then the matrix 
$$\begin{bmatrix}
\alpha & \beta \\ -b & a
\end{bmatrix}$$
is unimodular, hence the linear transformation associated to it is a lattice isomorphism that takes $(a,b)$ to $(1,0)$.
So we may assume that one of the rays of $K_A$ is the positive part of the $x$ axis. Now the image of the second cone generator is some $(c,d)$. It is easy to see that $d$ cannot be zero, and we may assume that it is positive (as we can always change the sign on the second coordinate of our map). Applying the unimodular transformation given by
$$\begin{bmatrix}
1 & \gamma \\ 0 & 1
\end{bmatrix}$$
we fix $(1,0)$ and move $(c,d)$ to $(c+\gamma d, d)$. By choosing the smallest integer~$\gamma$ that makes $c+\gamma d \geq 0$, we can make the first coordinate nonnegative and smaller than or equal to the second. We have just shown the following.
\begin{obs} \label{obs:canonical}
In Problem \ref{q:cone} we can always assume $K \subseteq \RR_+^2$ and is generated by $(1,0)$ and $(c,d)$ with $0 \leq c < d$.
\end{obs}

\begin{ex}
Let us consider the example in Figure \ref{fig:graphicexample2}, coming from the case of Example \ref{ex:graphicexample2} where $n=4$ so that we are considering the matrix
$$\begin{bmatrix}
5 & 4 & 3 \\
4 & 4 & 4 \\
3 & 4 & 5
\end{bmatrix}.$$
One of the generators for the cone is the vector $(1,1)$. To send it to $(1,0)$ we take the unimodular transformation $(x,y) \mapsto (x,y-x)$. This sends the other generator, $(-1,1)$, to the point $(-1,2)$. We now compose it with the unimodular transformation $(x,y) \mapsto (x+y,y)$, which will send it to $(1,2)$. 

The composition of the two maps gives us the transformation $(x,y) \mapsto (y,y-x)$ and when we apply it to the full diagram we get the canonical form shown in Figure \ref{fig:graphicexample3}.
\begin{figure}
    \centering    \includegraphics[width=0.45\linewidth]{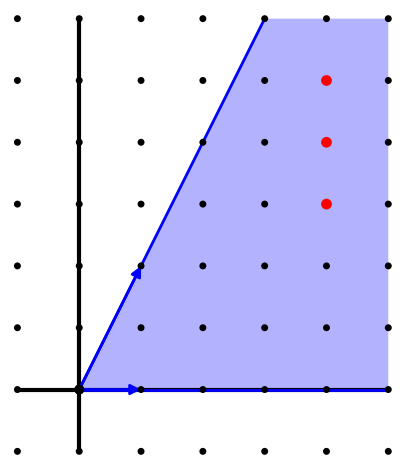}
    \caption{Transformation of Figure \ref{fig:graphicexample2} to a canonical form.}
    \label{fig:graphicexample3}
\end{figure}
\end{ex}

A source of ambiguity remains in this canonical form, since we can pick either of the two minimal generators of the cone to map to $(1,0)$, resulting in two different diagrams. Another thing to point out is that while the existence of such a form will be very useful to streamline arguments, there may be occasions where a different form of the diagram is more suitable. Example \ref{ex:graphicexample2} is a case where the canonical embedding is less amenable to work with than the initial embedding we used.

\subsection{Reduction to the $3\times 3$ case}
\label{subsec:reduction}

We saw in Section \ref{subsec:geometric} how to go from the rank2 problem to the problem of finding two generators in a rational cone for a finite set of points. We will start this section by trying the reverse: given a rational cone $K \subseteq \RR^2$ and a set of $m$ integer points $a_1,...,a_m$ in it, we want to construct a matrix that has nonnegative integer rank $2$ if and only if there are two integer points in $K$ that nonnegatively span $a_1,...,a_m$.

One interesting aspect is that while the number of columns of the matrix must correspond to the number $m$ of points, there is no such direct implication for the number of rows. We need to map the points $a_i$ linearly to some $\ZZ_+^n$ in such a way that their span intersected with $\RR_+^n$ is the image of the cone $K$ by the same linear map. 

\begin{ex} We will start by considering the example in Figure \ref{fig:graphicexample}. Our goal is to try to recover a matrix simply from the figure data. A first attempt would be to take the integer inequalities defining the faces of the cone $K$, 
$$h_1(x,y)=3x-y \geq 0, \ \ \ h_2(x) = y \geq 0;$$
and evaluate them at the three red points, obtaining the matrix
$$
\begin{bmatrix}
1 & 3 & 9 \\
2 & 0 & 3
\end{bmatrix}.
$$
However, this matrix has obviously nonnegative integer rank $2$, since the columns are generated by $(1,0)$ and $(0,1)$. The issue is that $(1,0)$ corresponds to the image of a non-integer point from $K$, the point $(1/3,0)$. We would like the integer points in the column space of our matrix to be images of points in $\ZZ^2$. 
In this case this can easily be achieved by adding a third, redundant, valid inequality
$$h_3(x,y)=x \geq 0,$$
adding a new row to our matrix and obtaining
$$
\begin{bmatrix}
1 & 3 & 9 \\
2 & 0 & 3 \\
1 & 1 & 4
\end{bmatrix}.
$$ 
In this case, we can see that we recovered precisely the original matrix from Example \ref{ex:graphicexample}, only with permuted rows.
\end{ex}

\begin{prop} \label{prop:3xn} For a rational pointed cone with nonempty interior $K \subseteq \RR^2$ and $m$ integer points $a_1,...,a_m \in K$, there is a $3 \times m$ matrix $B$ such that $\rankzn(B)=2$ if and only if there are two integer points in $K$ such that every $a_i$ is a nonnegative combination of them.
\end{prop}
\begin{proof}
To simplify the argument, note that by \Cref{obs:canonical} we can assume that $K$ is contained in the nonnegative quadrant and one of its rays is generated by $(1,0)$. 
One of the defining inequalities of $K$ will therefore be $y \geq 0$, the other will be $ax + by \geq 0$, for some integers $a,b$. We also know that $x \geq 0$ is valid on $K$. Evaluating these three linear forms on the points $a_1,...,a_m$, we will then give a $3 \times m$ nonnegative integer matrix $B$.

Since the rows of $B$ are given by the evaluations of $x,y$, and $ax+by$ on $m$ points, every point in $\col(B)$ will have the form $(x,y,ax+by)$. Thus, it is clear that the map $(x,y) \mapsto (x,y,ax+by)$ is a lattice isomorphism from $\ZZ^2$ to $\col(B) \cap \ZZ^3$. Moreover, a point in $\col(B)$ is nonnegative if and only if its preimage lies in $K$. So the cone $K$ together with the set of points $a_1,...,a_m$ is a diagram for the matrix $A$ as intended.
\end{proof}

We already knew how to go from the rank$2$ problem for an $n\times m$ matrix $A$ to Problem \ref{q:cone}, with $m$ integer points. Proposition \ref{prop:3xn} now tells us that we can go from that problem to a the rank$2$ problem on a $3 \times m$ matrix. We want to interpret this reduction directly in terms of matrices, bypassing the geometric construction. The general underlying idea is the following.

\begin{lem} \label{lem:matrix3n}
Let $A$ and $B$ be rank $r$ nonnegative integer matrices of size $n \times m$ and $k \times m$, respectively. If $A$ and $B$ satisfy the following conditions:
\begin{enumerate}[label = (\roman{enumi})]
    \item $A$ and $B$ have the same row space,
    \item $Ax \in \ZZ^n$ if and only if $Bx \in \ZZ^k$, and 
    \item $Ax \geq 0$ if and only if $Bx \geq 0$,
\end{enumerate}
then either $\rankzn(A)=\rankzn(B)=r$ or they are both greater than $r$.
\end{lem}
\begin{proof}
Define a lattice isomorphism $\varphi$ from $\col(A) \cap \ZZ^n$ to $\col(B) \cap \ZZ^k$ by sending the $i$-th column of $A$ to the $i$-th column of $B$, and extend it linearly, i.e., for $Ax \in \ZZ^n$, $\varphi(Ax)=Bx$.

This map is well-defined: $Ax \in \ZZ^n$ implies $Bx \in \ZZ^k$ by hypothesis, while $Ax=Ay$ means that $x-y$ is orthogonal to the row space of $A$, which is the same as that of $B$, hence $Bx=By$. By symmetry, it is actually an isomorphism, with the inverse being the linear map that sends the $i$-th column of $B$ to the $i$-th column of $A$. Moreover $\varphi$ and its inverse send nonnegative elements to nonnegative elements by the third condition.

Now note that $\rankzn(A)=r$ if and only if $A=MN$ where $M$ and $N$ are nonnegative integer matrices and $M$ is $n \times r$. Note that the columns of $M$ are in the column space of $A$, hence we can apply $\varphi$ to each of them to get a new matrix $M' \in \ZZ_+^{k \times r}$. The linearity of $\varphi$ then immediately tells us that $M'N = B$ which tells us that $\rankzn(B)=r$. The same argument gives us the reverse implication. 
\end{proof}

The argument in Proposition \ref{prop:3xn} boils down to saying that for every $n \times m$ nonnegative integer matrix $A$ we can construct a $3 \times m$ nonnegative integer matrix $B$ satisfying the conditions of Lemma \ref{lem:matrix3n}. In what follows we show how to construct such a matrix explicitly.

To guarantee that conditions (i) and (ii) are satisfied, we must pick the rows of $B$ in such a way that they generate the same lattice as the rows of $A$, $\Lambda_{\row}(A)$. We will start by picking two rows that will ensure that the third condition holds. They should correspond to the inequalities that cut out the cone $\col(A)\cap \RR_+^n$, so we will start by identifying the two extreme rays of this cone. Because we are working over a two dimensional space we can do this easily, for instance, by taking two columns of $A$ that are not linearly dependent, $A_1$ and $A_2$, and considering the ray spanned by $A_\lambda=(1-\lambda)A_1 + \lambda A_2, \lambda \in \RR$. For $\lambda \in [0,1]$ this is in the positive orthant, but has $\lambda$ increases beyond $1$ some entry $i$ will eventually turn from positive to zero at some value $\lambda^+$. Similarly, as $\lambda$ decreases further than $0$, at some point a positive entry $j$ will become zero at some $\lambda^-$. The rays spanned by $A_{\lambda^-}$ and $A_{\lambda^+}$ are the extreme rays of the cone $\col(A)\cap \RR_+^n$. Moreover, the row $i$ (and row $j$) of $A$ correspond to the evaluation at each column of $A$ of a linear form that vanishes on $A_{\lambda^+}$ (respectively $A_{\lambda^-}$). We could add these two rows $A_i$ and $A_j$ directly to $B$, but in order to help to verify condition (ii), we will first divide each row by the largest positive integer such that the row is still in $\Lambda_{\row}(A)$. Thus, we obtain our first two rows of $B$, $b_1$ and $b_2$, where both are primitive elements of $\Lambda_{\row}(A)$.

With this we already guarantee $\row(A)=\row(B)$ and $Ax \geq 0$ if and only if $Bx \geq 0$.  For condition (ii), by construction we have the lattice generated by the rows of $B$, $\Lambda_{\row}(B)$,  contained in $\Lambda_{\row}(A)$. The only question remaining is if this lattice is strictly smaller than $\Lambda_{\row}(A)$.

To ensure this is not the case, we want to find an integer element $b_3$ in $\Lambda_{\row}(A)$, obtained as a positive rational combination of $b_1$ and $b_2$,  that together with $b_1$ and $b_2$ spans that lattice, and add it to $B$. This is always possible, since any primitive element of a rank 2 lattice can be extended to a basis of said lattice by adding some element. By adding multiples of $b_1$ and $b_2$ to this new element, we can then guarantee that it becomes a positive combination of $b_1$ and $b_2$, without changing the lattice spanned.

To do this in practice, find a basis $a_1,a_2$ for the lattice $\Lambda_{\row}(A)$, by using the LLL algorithm for instance, and write $b_1$ in terms of that basis, obtaining $(p,q)$. The primitivity of $b_1$ implies $\textup{gcd}(p,q)=1$, so by using the extended euclidean algorithm, we can find $r,s \in \ZZ$ such that $rp-sq=1$. Then $\overline{b_3}=sa_1+ra_2$ spans the lattice $\Lambda_{\row}(A)$ together with $b_1$. By adding positive multiples of $b_1$ and $b_2$ we can replace $\overline{b_3}$ by a $b_3$ that is a positive combination of $b_1$ and $b_2$, such that $\{b_1,b_2,b_3\}$ span the lattice as intended. It is now simple to check that this matrix $B$ satisfies all the conditions of \Cref{lem:matrix3n}.

Note that the construction of $M$ is computationally simple. The only nontrivial steps of the procedure are the use of LLL and the extended euclidean algorithm, both of which are polynomial in the bitlengths of the numbers, and on the sizes of the matrices involved.

We can now leverage this construction, applying it twice. If we have an $n \times m$ matrix of rank $2$ fow which we want to solve the minrank problem, we can apply Proposition \ref{prop:3xn} to obtain a $3 \times n$ matrix with the same solution. Now, since transposing a matrix does not change its ranks, we can apply the proposition again to the transpose of this matrix, attaining a $3 \times 3$ matrix that will also have the same solution as the original one. We thus obtain the following fact.

\begin{thm}
Given a rank $2$ nonnegative integer matrix $A \in \ZZ_+^{n \times m}$, we can always construct a $3 \times 3$ nonnegative integer matrix $B$ of rank $2$ whose nonnegative integer rank is $2$ if and only if the nonnegative integer rank of $A$ is $2$. Moreover, the computational effort and the norms of the entries of $B$ are polynomial in $n,m$ and $\log(\max_{i,j} |A_{ij}|)$.
\end{thm}

\begin{remark} One could think that a similar geometric reduction would work for other minrank problems, not necessarily just rank $2$. However, this construction makes key use of the fact that the cone $\col(A) \cap \RR_+^n$ is two dimensional, and any two dimensional cone is defined by at most two inequalities. For any other rank, the number of inequalities defining $\col(A) \cap \RR_+^n$ can be $n$, which makes this geometric approach useless.
\end{remark}

\begin{ex}
We will explore this reduction procedure with a $5 \times 5$ example.
Consider the rank $2$ matrix
$$
A=\begin{bmatrix}
 2 & 4 & 6 & 4 & 2 \\
 4 & 7 & 10 & 5 & 2 \\
 5 & 8 & 11 & 4 & 1 \\
 2 & 6 & 10 & 10 & 6 \\
 3 & 7 & 11 & 9 & 5 \\
\end{bmatrix}.
$$
Its diagram is the first in Figure \ref{fig:graphicexample4}. 

\begin{figure}
    \centering    \includegraphics[width=0.35\linewidth]{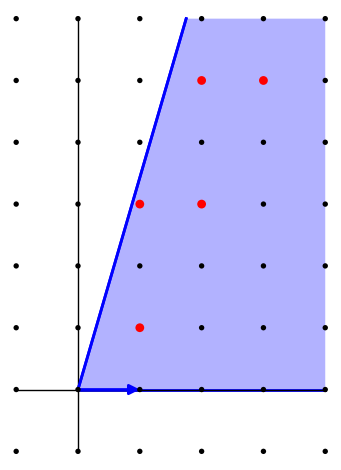}  \hspace{2cm} \includegraphics[width=0.35\linewidth]{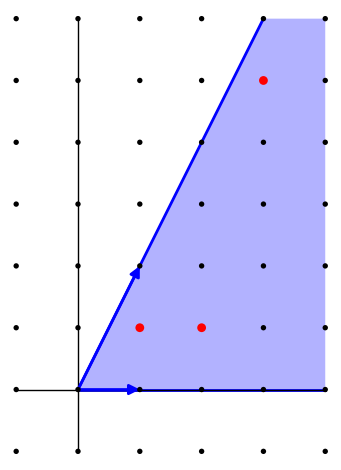}
    \caption{Diagrams of a $5 \times 5$ matrix and its reduction.}
    \label{fig:graphicexample4}
\end{figure}

An LLL-reduced basis for the $\Lambda_{\row}(A)$ can be computed to be
\begin{equation}
 \begin{bmatrix}
 1 & 1 & 1 & -1 & -1 \\
 0 & -1 & -2 & -3 & -2 \\
\end{bmatrix}.  \label{latticebasis} 
\end{equation}

By picking $A_1$ and $A_2$ the first two columns, we notice that $3A_1-A_2=(2,5,7,0,2)^\top$ and $5A_2-8A_1=(4,3,0,14,11)^\top$ are the integer generators of the cone $\textup{col}(A) \cap \RR_+^5$. The zeros are in rows $3$ and $4$ so we will start building $B$ by taking each of these rows of $A$ divided by the largest common divisor of its entries such that the result is still in $\Lambda_{\row}(A)$:
$$B'=\begin{bmatrix}
 5 & 8 & 11 & 4 & 1 \\
 1 & 3 & 5 & 5 & 3 \\
 \end{bmatrix}.$$

In terms of the basis in \eqref{latticebasis} before these rows correspond to vectors $(5,-3)$ and $(1,-2)$. Taking the first vector as $b_1$ we have that $5\times (-1) - (-3)\times2 = 1$, hence we get $\overline{b_3}=(2,-1)$ in terms of our basis, which corresponds to $(2,3,4,1,0)$. In terms of $b_1$ and $b_2$ this is $\frac{1}{7}(3b_1-b_2)$ so we
need to add $b_2$ to make it a positive combination, and we get $b_3=(3,6,9,6,3)$. We could use it as is, or make it primitive. Since its coordinates in the basis of the lattice are $(3,-3)$, we can divide it by $3$ while keeping it in the lattice, so we obtain the 
$3 \times 5$ matrix
$$
B=\begin{bmatrix}
 5 & 8 & 11 & 4 & 1\\
 1 & 3 & 5 & 5 & 3 \\
 1 & 2 & 3 & 2 & 1
 \end{bmatrix}.
$$
Notice that these three rows are very much related to the original matrix: they are the third, fourth and first rows of $A$ up to multipliers. The diagram of (the transpose) of this matrix is the second one in Figure \ref{fig:graphicexample4}. 

Repeating this operation on the transpose of $B$, we pick the two last rows $B_2$ and $B_3$ and note that $B_2-B_3=(0,1,2,3,2)$ and $3B_3-B_2=(2,3,4,1,0)$ generate $\textup{row}(B) \cap \RR_+^5$, so we need the first and last columns of $B$, which are obviously primitive
$$
C'=\begin{bmatrix}
 5 & 1  \\
 1 & 3  \\
 1 & 1  
 \end{bmatrix}.
 $$
 Using LLL we obtain the basis $(-2,1,0)^\top$, $(1,3,1)^\top$ for the column lattice. Our columns are then $(-2,1)$ and $(0,1)$ in this basis. It is clear that the basis element $(-2,1,0)^\top$ completes the basis, but since
 $(-2,1,0)^\top$ is $\frac{1}{2}(c_2-c_1)$ we have to add it $c_1$ to make it an integer combination, obtaining $c_3=(3,2,1)^\top$, which is primitive.
This gives us the $3 \times 3$ matrix
$$
C=\begin{bmatrix}
 5& 1 & 3 \\
 1 & 3 & 2 \\
 1 & 1 & 1 
 \end{bmatrix}.
$$
\end{ex}

\begin{ex}
The fact that we can reduce checking if nonnegative integer rank equals 2 to the case of $3 \times 3$ matrices could suggest that, in analogy to the usual rank, if every $3 \times 3$ submatrix has nonnegative integer rank $2$ than so does have the full matrix. However, as the following matrix shows, this is wrong. In Figure \ref{fig:graphicexample5} we can see the diagram of the matrix
$$  
\begin{bmatrix}
0 & 6 & 10 & 15 \\
1 & 3 & 5 & 8   \\
5 & 9 & 15 & 25
\end{bmatrix}.
$$
What happens is that every set of three red points can be spanned by two points in the cone, but the four of them can not.

\begin{figure}
    \centering    \includegraphics[width=0.3\linewidth]{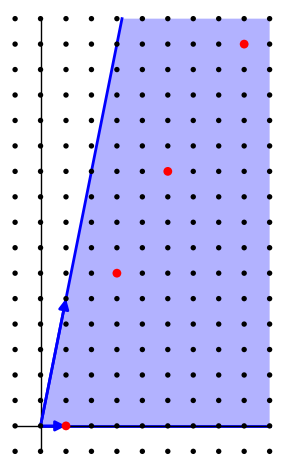}
    \caption{Diagram of a $3 \times 4$ matrix.}
    \label{fig:graphicexample5}
\end{figure}

To see it in the matrix just note that any pair of  generators for the column lattice would need to include one with zero on the first entry. Now the largest common divisor of the first entries of the remaining three columns is $1$, and there is no integer column in the column span with one on its first entry, so you would need $3$ generators.

\end{ex}

\section{Algorithm}
\label{sec:alg}

In the previous section we saw that checking if an $n \times m$ matrix of rank $2$ also has nonnegative integer rank $2$ can be reduced to checking it for a $3 \times 3$ matrix. Unfortunately, this does not get us any closer to solving this problem, it just tells us that the $3 \times 3$ version of the problem is as hard as the original one. In this section, we present an algorithm that actually solves the rank2 problem. 

A first approach to the problem would be a brute force search: given a matrix $A$, we are looking for two integer points in the 2-dimensional cone of its diagram that generate all the points corresponding to columns of $A$ as nonnegative combinations (see Problem~\ref{q:cone} and the discussion around it). Since each generator has to be in the nonnegative orthant and be entrywise smaller than or equal to those points, we have a finite search space that we could exhaust.

This approach certainly works, but is impractical due to the size of the search space. The main idea of our proposed algorithm is to use the geometric interpretation of \Cref{subsec:geometric} to restrict the search space for nonnegative integer vectors that factorize the given matrix (see Algorithm~\ref{alg:NNI2}). The intuition for reducing the search space can be seen in Example~\ref{ex:graphicexample2}, where we note that in order to generate the points of $\col(A)$ nonnegatively, our 2 generators must somehow live ``outside'' this set of points. We formalize this idea below. 

\subsection{Algorithm details}

The algorithm receives a rank $2$ nonnegative matrix $A$, of size $n \times m$ and proceeds according to the following steps.

\subsubsection{Constructing the diagram}

Our first step is to construct the canonical diagram associated to matrix $A$. To do this we need to construct a basis for the lattice 
$\col(A)\cap\ZZ_+^n$ and map it to $\ZZ^2$. This can be done by computing the Smith Normal form of $A$ (see \cite{Micc08}). After that we can perform a unimodular transformation to obtain a canonical form as discussed in \Cref{subsec:geometric}.

We obtain a cone $K_A$, the image of the cone $\col(A) \cap \RR_+^2$, that will be spanned by $e_1$ and some $\bm{c} \in \ZZ_+^2$, and a set of integer points $S_C=\{\bm{c}_1,...,\bm{c}_m\}$ inside $K_A$ that are images of the columns of $A$.

\subsubsection{Decomposing the Cone}

We take the cone $K \subseteq K_A$ generated by $S_C$ and find its generators $\bm{u}$ and $\bm{v}$ (these are just the points among the $S_C$ that have the lowest and highest slope).
 Thus we can decompose $K_A$ as the union of $K$ and an ``upper'' and ``lower'' cone 
\[
K_A = K_+\cup K\cup K_-,
\]
as pictured in Figure~\ref{fig:cone decomp}.
Note that in canonical form, we can write $K_{-} = \cone(\bm{e}_1,\bm{u})$, $K = \cone(\bm{u},\bm{v})$, and $K_{+} = \cone(\bm{v},\bm{c})$.

We claim that if $\rankzn(A)=2$, then the two lattice points $\bm{a},\bm{b}\in K_A$ that generate the points corresponding to the columns of $A$ must satisfy $\bm{a}\in K_{-}$ and $\bm{b}\in K_{+}$. To see this, notice that we need to generate $\bm{u},\bm{v}$ which themselves generate the cone $K$. If $\bm{u},\bm{v}$ are both nonnegative integer combinations of $\bm{a},\bm{b}$ then, in particular, $\bm{u},\bm{v}\in\cone(\bm{a},\bm{b})$. Thus $K\subset \cone(\bm{a},\bm{b})$ and without loss of generality, we have $\bm{a}\in K_{-}$ and $\bm{b}\in K_{+}$, as claimed.

\definecolor{Grn}{cmyk}{1,0.1,0.9,0.6}
\begin{figure}
    \centering
    \scalebox{0.8}{
    \begin{tikzpicture}
        \node[anchor=south west,inner sep=0] at (0,0) {\includegraphics[width=0.5\linewidth]{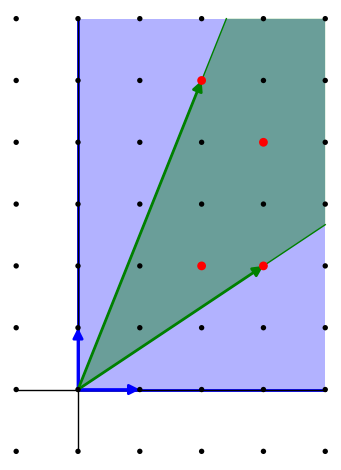}};
       \node at (4,5) {\Large\textcolor{Grn}{$K$}}; 
       \node at (5,2) {\Large\textcolor{Blue}{$K_{-}$}}; 
       \node at (5,3.3) {\Large\textcolor{OliveGreen}{$\bm{u}$}}; 
       \node at (2,6) {\Large\textcolor{Blue}{$K_{+}$}}; 
       \node at (3.2,7) {\Large\textcolor{OliveGreen}{$\bm{v}$}}; 
       \node at (1,2.1) {\Large\textcolor{blue}{$\bm{c}$}};  
    \end{tikzpicture}
    }
    \caption{Decomposition of the cone $K_A$.}
    \label{fig:cone decomp}
\end{figure}
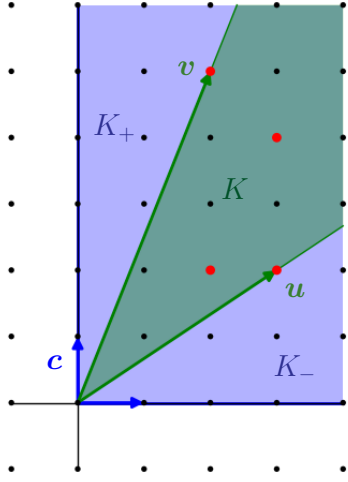 




\subsubsection{Bounding the Search Space}
\label{sssec:alg_gens}

Next we show that there is a particular bounded region of each of $K_+$ and $K_{-}$ that must contain the generators $\bm{a},\bm{b}$, which allows us to search by enumerating all the integer points in the region. 

Recall that $\bm{u}\in K\cap K_{-}$ is one of the points we need to generate. Suppose that 
$\bm{u} = k\bm{a}+\ell\bm{b}$ for $k,\ell\in\mathbb{Z}_+$. Then we get that $k\bm{a} = \bm{u}-\ell\bm{b}\in K_{-}\cap(\bm{u}-K_{+})$. Hence for each integer point in the triangle $\Delta = K_{-}\cap(\bm{u}-K_{+})$, we get a possible $\bm{a}$ by choosing the minimal lattice point on the ray through that point (see Figure~\ref{fig:triangle_search}). 

\begin{figure}
    \centering  
    \scalebox{0.8}{
    \begin{tikzpicture}
        \node[anchor=south west,inner sep=0] at (0,0) {\includegraphics[width=0.5\linewidth]{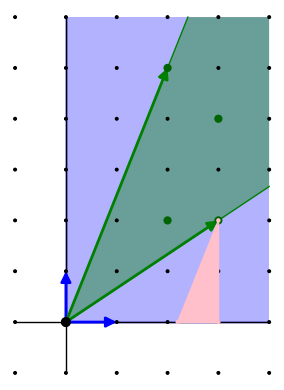}};
        \node at (4.4,2.15) {\Large\textcolor{Red}{$\Delta$}};
        \node[circle, fill, red, inner sep = 0pt, minimum size=5pt] at (4.67,3.63) {};
        \node[circle, fill, red, inner sep = 0pt, minimum size=5pt] at (4.67,2.52) {};
        \node[circle, fill, red, inner sep = 0pt, minimum size=5pt] at (4.67,1.46) {};
        \node at (4.67, 1.15) {\Large\textcolor{Red}{$k\bm{a}$}};
        \node[circle, draw, red, inner sep = 0pt, minimum size=5pt] at (2.5,1.46) {};
        \node at (2.6, 1.15) {\Large\textcolor{Red}{$\bm{a}$}};
    \end{tikzpicture}
    }
    \caption{Triangle of possible $k\bm{a}$ vectors}
    \label{fig:triangle_search}
\end{figure}
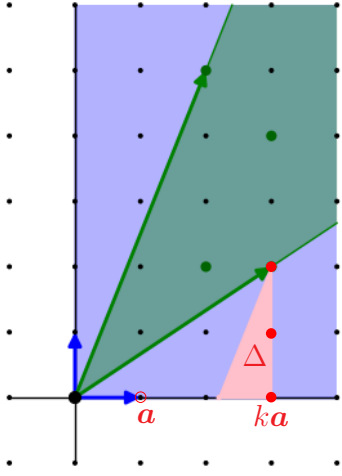

For each possible choice of $\bm{a}$, we then search for an appropriate $\bm{b}$. Since $\ell\bm{b}=\bm{u}-k\bm{a}$, if $k\bm{a}\neq \bm{u}$, then $\bm{u}-k\bm{a}$ determines the ray on which $\bm{b}$ must live and we can take $\bm{b}$ to be the minimal lattice point on that ray (which necessarily lives in $K_+$ since $k\bm{a}\in\bm{u}-K_+$). 

If $k\bm{a} = \bm{u}$, then 
we can look at the point $\bm{v}\in K\cap K_{+}$, which also has to be generated by $\bm{a}$ and $\bm{b}$. But $\bm{v}= k'\bm{a}+\ell'\bm{b}$ for $k',\ell'\in\mathbb{Z}_+$ implies that we can take as $\bm{b}$ any of the possible generators of the rays spanned by the integer points in the ray $\bm{v}-r\bm{a}$.

This procedure generates a list of candidates $(\bm{a},\bm{b})$ for generators. For each we must check if they indeed generate all the points $\bm{c}_i$ corresponding to the columns of $A$. If any of them do, we have found that the nonnegative integer rank of $A$ is two, if none of them do, the nonnegative integer rank of $A$ must be higher.

Note that checking if $\bm{c}_i$ is generated as an integer combination of $\bm{a}$ and $\bm{b}$ is simply checking if the coefficient vector 
\[
\bm{w}_i = \begin{bmatrix}
    \bm{a}  & \bm{b}
\end{bmatrix}^{-1}\bm{c}_i.
\]
is nonnegative and integer.


\subsubsection{Factorization from generators}

If we do find generators $\bm{a}$ and $\bm{b}$ we then can optionally proceed to reconstruct the matrix factorization. Note that the map that sends $\col(A) \cap \ZZ_+^n$ to $\ZZ_+^2$ has an inverse that can be written as $\phi(x)=Bx$ where the columns of $B$ are simply the preimages of the vectors $(1,0)$ and $(0,1)$. 

We claim that $A=(B\begin{bmatrix}
    \bm{a} & \bm{b}
\end{bmatrix})\begin{bmatrix}
    \bm{w}_1 & \cdots & \bm{w}_m
\end{bmatrix}$ provides a nonnegative integer rank $2$ factorization of $A$. To see this just note that $$\begin{bmatrix}
    \bm{a} & \bm{b}
\end{bmatrix}\begin{bmatrix}
    \bm{w}_1 & \cdots & \bm{w}_m
\end{bmatrix} = \begin{bmatrix}
    \bm{c}_1 & \cdots & \bm{c}_m
\end{bmatrix}$$
and the preimage of each $\bm{c}_i$ is the $i$-th column of $A$ so we indeed have the equality. Moreover, the $\bm{w}_i$ are nonnegative integer by hypothesis, and $B\begin{bmatrix}
    \bm{a} & \bm{b}
\end{bmatrix}$ is also nonnegative and integer as both $\bm{a}$ and $\bm{b}$ lie in $K_A \cap \ZZ^2$ so their preimages are in $\col(A) \cap \ZZ_+^n$.

\begin{ex}\label{eg:beasley_factor}
    Take 
    \[
    A=\begin{bmatrix} 2 & 0 & 3 \\ 1 & 1 & 4 \\ 1 & 3 & 9 \end{bmatrix}.
    \]
    The column space of $A$ is generated by $(1,0,-1)^\top$ and $(0,1,3)^\top$ and is orthogonal to the subspace generated by $(1,-3,1)^\top.$ From this we see that $\col(A)\cap\RR^3_+$ has generators $(3,1,0)^\top$ and $(0,1,3)^\top$. 

    Following $\cite{Micc08}$, we get that the lattice $\col(A)\cap\ZZ^3$ is generated by 
    \[
    B = \begin{bmatrix}
        0 & 1 \\ 1 & 0 \\ 3 & -1\end{bmatrix}
    \]
    and the coordinates of the columns of $A$ in this basis are $\bm{c}_1 = (1,2)^\top, \bm{c}_2 = (1,0)^\top$ and $\bm{c}_3= (3,4)^\top$. Furthermore, we can write the generators of $K_A$ with coordinates $(1,3)^\top$ and $(1,0)^\top$.
The resulting cones are pictured in Figure~\ref{fig:Beasley matrix cones}. Since they are already in canonical form, we have $T = I_2$, the identity matrix. 
        
    Now the triangle $\Delta = K_{-}\cap (\bm{u}-K_+)$ is simply the point $\bm{u} = (1,0)^\top$. Thus we can only have $\bm{a}=(1,0)^\top$. Then the only lattice point of the form $\bm{v}-k\bm{a}$ in $K_A$ is $\bm{v} = (1,2)^\top$, hence $\bm{b}=\bm{v}$. But now 
    \[
    \bm{w}_3 = \begin{bmatrix}
        1 & 1 \\
        0 & 2\\
    \end{bmatrix}^{-1}\begin{pmatrix}
        4 \\ 3
    \end{pmatrix} = \begin{pmatrix}
        \frac{5}{2} \\[2pt] \frac{3}{2} 
    \end{pmatrix}\notin\mathbb{Z}^2,
    \]
    so that $\rankzn(A)\neq 2.$
    
    \begin{figure}
        \centering        \includegraphics[width=0.5\linewidth]{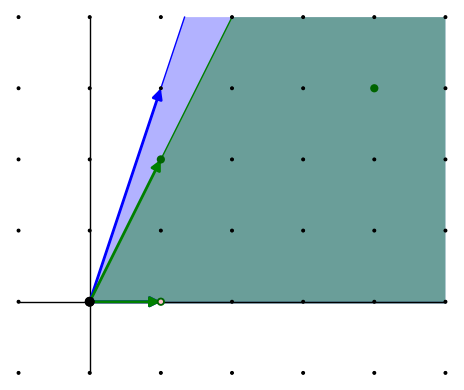}
        \caption{Cones of Example~\ref{eg:beasley_factor}: $K_+$ in blue, $K$ in green, and $K_-$ is just the $x$-axis.}
        \label{fig:Beasley matrix cones}
    \end{figure}    
\end{ex}

The complexity of this algorithm is theoretically not much better than full exhaustive search, as the step of searching for candidate factors in the triangle $\Delta$ is polynomial in $\|u\|$ which in turn is polynomial in $\|A\|_{\infty}$ and the size of the matrix. Running the code for the Example \ref{ex:graphicexample2} with growing $t$ we see that in fact the time taken seems to grow faster than linearly with the size of the entries (see \Cref{fig:n_mat_factor}).

\begin{remark}
This algorithm cannot easily be extended to ranks higher than two. The main reason for this is that both the cone $K$ and the cone $K_A$ get much more complicated in higher dimensions, and the containment between them is highly nontrivial. If the rank was $k$, our analogous candidate factors would be integer points $\bm{a}_1,...,\bm{a}_k$ in $K_A$ that span a simplex that contains $K$, but there is no simple way of searching over these in the same way the cone decomposition allows us in rank $2$.
\end{remark}

\subsection{Numerical Testing}

To demonstrate the effectiveness of our algorithm, we generate a collection of rank 2 nonnegative integer matrices to factor. We compare the average time to factor the full matrix versus the average time to reduce to a $3\times 3$ representative as described in Section~\ref{subsec:reduction} for various matrix sizes. We also record the percentage of test matrices that are found to have nonnegative integer rank~2. All algorithms are implemented in SageMath \cite{sage} and run on a laptop (2020 MacBook Air, M1 chip). The runtimes recorded in the following tests are meant to be indicative of general trends and as a proof of concept of the algorithm; they should not be taken as benchmarks and may vary with implementation, which we have not attempted to optimize. The code is available by request from the second author.

We generate two sets of test matrices. The first is a random collection of matrices of some fixed size $n\times n$ for a selection of values of $n$ with increasing entry size. The second is a collection of $3\times 3$ matrices with average entry size distributed roughly uniformly between 3 and 100 whose entries are all close to the average. 

To generate the first set of test matrices, we use a function\footnote{ DiscreteGaussianDistributionLatticeSampler}
in SageMath that returns lattice vectors $x$ proportionally to $\exp(-|x|^2/(2\sigma^2))$ for a chosen lattice and parameter $\sigma$. In our case, the lattice is $\ZZ^2$ and we vary $\sigma$, which has the effect of increasing the size of the entries of our matrix as we increase it. If we want an $n\times m$ matrix, we first generate $m$ nonzero lattice vectors in the nonnegative orthant to be columns of a matrix $C$. Then we generate additional nonzero lattice vectors and add them as rows of a matrix $B$ if they live in the dual of $\cone(\cols(C))$ until $B$ has $n$ rows. The resulting product $A=BC$ will be a nonnegative integer matrix of (regular) rank 2. 

Table~\ref{tab:numerics} shows the results of numerical testing on 1000 matrices with $n\in\{3,5,10,50\}$ and 100 matrices with $n=100$ with each of $\sigma\in\{3,6,10,25\}$. For each choice of parameters, we list the average largest entry over the set of matrices, the average number of seconds to factorize each matrix, and the number of matrices that have nonnegative integer rank 2. 

From the tables, we can see that there is a general correlation between the size of the largest entry and the factorization time. Similarly, there seems to be some correlation between $n$ and factorization time; take, for example, $n=3,\sigma=10$ versus $n=50,\sigma=6$ which have roughly the same average largest entry, but $n=3$ is around 10 times faster on average. 

    \begin{table}[h]
    \caption{Testing Algorithm~\ref{alg:NNI2} on random $n\times n$ rank 2 matrices.}
    \label{tab:numerics}
    \centering
    \begin{tabular}{cc|c|c|c|c|c}
    \hline\hline
    \multicolumn{2}{c|}{Parameters} & Avg largest & \multicolumn{3}{|c|}{Seconds per matrix} & NNI rank 2     \\
    $n$ & $\sigma$ & matrix entry & min & average & max & (count) \\ \hline
    3 & 3 & 24.4 & 0.00397 & 0.00550 & 0.05487 & 869/1000  \\
    3 & 6 & 95.9 & 0.00397 & 0.00675 & 0.41773 & 740/1000  \\
    3 & 10 & 275.1 & 0.00408 & 0.00998 & 0.51089 & 662/1000  \\
    3 & 25 & 1685.8 & 0.00457 & 0.08304 & 4.27659 & 619/1000 \\ \hline
    5 & 3 & 31.8 & 0.00492 & 0.00809 & 0.55515 & 729/1000  \\
    5 & 6 & 131.7 & 0.00507 & 0.00713 & 0.01644 & 529/1000  \\
    5 & 10 & 370.4 & 0.00518 & 0.00847 & 0.40386 & 461/1000  \\
    5 & 25 & 2321.2 & 0.00614 & 0.02747 & 1.25443 & 369/1000 \\ \hline
    10 & 3 & 44.6 & 0.00753 & 0.01169 & 0.07471 & 744/1000  \\ 
    10 & 6 & 174.3 & 0.00710 & 0.01031 & 0.02024 & 583/1000 \\
    10 & 10 & 503.2 & 0.00695 & 0.01042 & 0.01939 & 479/1000 \\
    10 & 25 & 3079.5 & 0.00805 & 0.01666 & 0.61764 & 354/1000 \\ \hline
    50 & 3 & 71.9 & 0.06883 & 0.09754 & 0.44215 & 998/1000 \\ 
    50 & 6 & 292.5 & 0.06677 & 0.08828 & 0.37194 & 973/1000 \\ 
    50 & 10 & 803.0 & 0.06836 & 0.08927 & 0.41484 & 848/1000 \\ 
    50 & 25 & 5034.9 & 0.07059 & 0.09617 & 0.41014 & 598/1000 \\ \hline  
    100 & 3 & 84.0 & 0.25288 & 0.89296 & 1.84873 & 100/100 \\
    100 & 6 & 331.4 & 0.27246 & 1.13803 & 1.86426 & 100/100 \\
    100 & 10 & 934.2 & 0.33168 & 1.74228 & 4.21948 & 100/100 \\
    100 & 25 & 6009.6 & 1.14274 & 2.80229 & 5.10633 & 80/100 \\
     \hline
    \end{tabular}
    \end{table}

   \begin{table}[h]
    \caption{Direct factorization vs. reduce then factor on $n\times n$ matrices.}
    \label{tab:reduction}
    \centering
    \begin{tabular}{c|c|c|c|c@{\;+}c|c}
    \hline\hline
    \multicolumn{2}{c|}{Parameters} & Avg largest & \multicolumn{3}{|c|}{Average seconds per matrix } & NNI rank 2     \\
    n & $\sigma$ & matrix entry & factor & reduce & factor & (count) \\ \hline
10 & 3 & 45.1 & 0.01062 & 0.04027 & 0.00653 & 77/100\\
10 & 6 & 177.6 & 0.01314 & 0.04226 & 0.00657 & 58/100\\
10 & 10 & 481.1 & 0.01136 & 0.04482 & 0.00675 & 39/100\\
10 & 25 & 3009.7 & 0.01816 & 0.08077 & 0.01396 & 34/100\\ \hline
1000 & 3 & 9994.3 & 38295.8 & 285.6 & 8.26 & 3/3\\
\hline
\hline    
     \hline
    \end{tabular}
    \end{table}

In Table~\ref{tab:reduction}, we record the average factorization times for 100 randomly-generated $10\times 10$ matrices, as well as for 3 randomly generated $1000\times 1000$ matrices. We also record the average time to reduce each matrix in the set to the $3\times 3$ case, as described in Section~\ref{subsec:reduction}, and the resulting average time to factor each $3\times 3$. Although there is some improvement in the factorization time from the $10\times 10$ case to its reduction, it is negligible in comparison to the time it takes to perform the reduction. However, at $n=1000$, we see that first performing the reduction significantly improves the overall factorization time.

For our second test set, we further investigate the relationship between the size of the matrix entries and the factorization time. Additionally, to control for the fact that when a matrix has nonnegative integer rank 2, factorization could be fast if we happen to find appropriate generators early, we attempt to construct matrices that are likely to have higher rank. To do this we take inspiration from Example~\ref{ex:graphicexample2}. In Figure~\ref{fig:n_mat_factor}, we see the effect of increasing $t$, the average entry size, on the factorization time for the $3\times 3$ matrix $B_t$, which has nonnegative integer rank 3.

\begin{figure}
    \centering
    \includegraphics[width=0.7\linewidth]{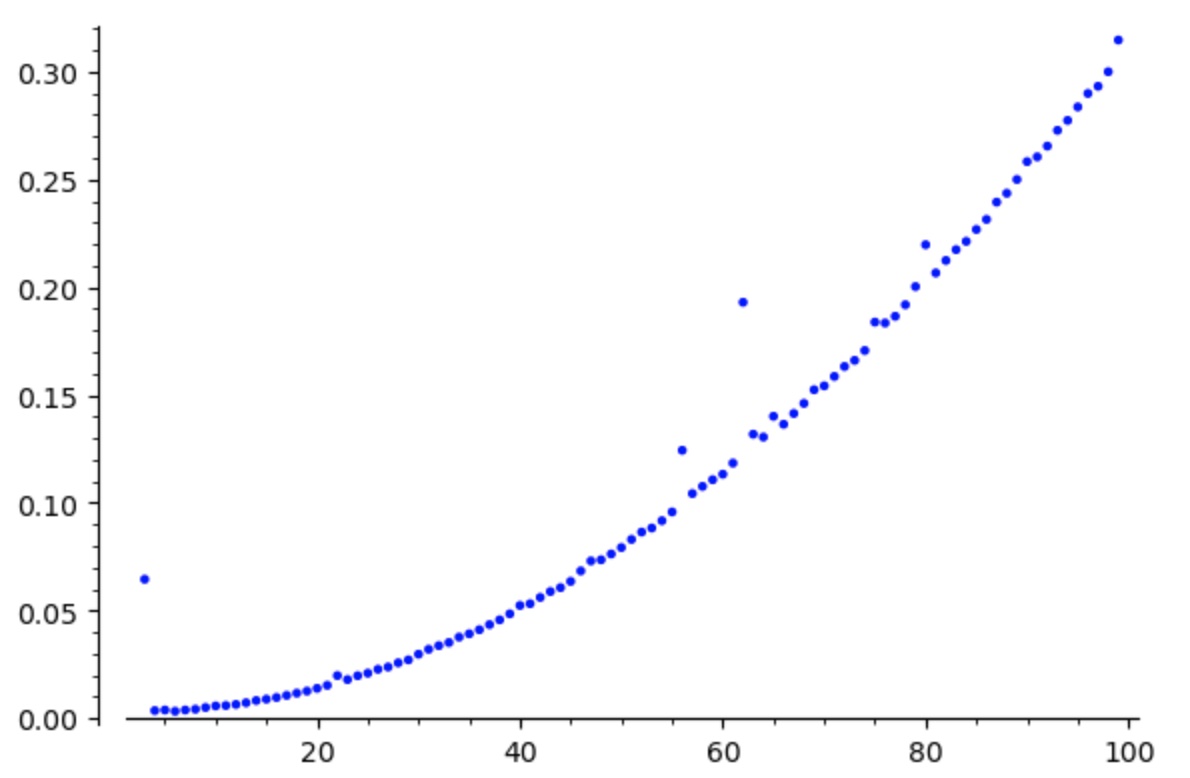}
    \caption{Factoring times for matrices of Example~\ref{ex:graphicexample2} for increasing $t$}
    \label{fig:n_mat_factor}
\end{figure}

We generate a similar set of 1000 $3\times 3$ matrices that have entries close to some positive integer $t$, which we choose from a uniform distribution on $[3,100]$. In particular, we use the same lattice distribution sampler as before with $\sigma=2$ centered at $(t,t)^\top$ to generate 3 points inside $\cone((1,0)^\top, (1,2)^\top)$ that form a diagram similar to that of Figure~\ref{fig:graphicexample3}. To obtain the corresponding matrix we evaluate the 3 points at $x,y,$ and $2x-y$.

In Figure~\ref{fig:numerics_n}, we plot the average value $x$ of the entries of each matrix versus its factorization time. Red stars indicate matrices of nonnegative integer rank~3, while blue circles indicate rank 2. 

\begin{figure}[h]
    \centering
    \includegraphics[width=0.9\linewidth]{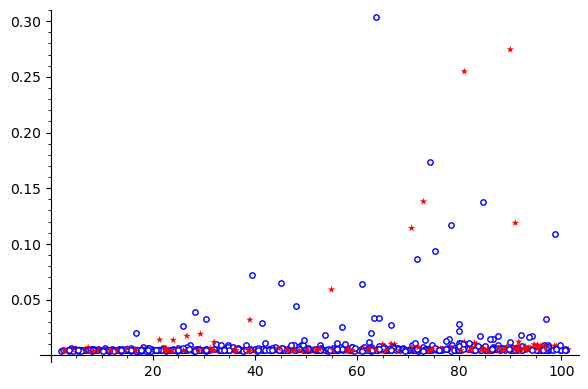}
    \caption{Factoring times for matrices with average entry $x$}
    \label{fig:numerics_n}
\end{figure}


While we see a slight upward trend in these figures, it is largely among outliers. The average factorization time among all matrices (of which 514 were rank 2) was 0.00812 seconds. The majority of matrices are simply factored too quickly to identify any meaningful trend. 

\section{Conclusion}

We interpret the nonnegative integer factorization problem geometrically, as finding a particular set of generators for  points in a rational cone. In this way we are able to introduce an algorithm that can efficiently answer \Cref{Q:rank2}.  
Even from our preliminary implementation of Algorithm~\ref{alg:NNI2}, we can see that it is possible to check if even large dense matrices have nonnegative integer rank~2. 
In the case of these large matrices, we can take advantage of our reduction to the equivalent $3\times 3$ problem to further improve the computational time. 

However, as we have seen, the current approach has serious limitations. The algorithm fundamentally relies on the fact that we are only considering matrices of (regular) rank 2. While much of our interpretation generalizes to higher ranks, many open questions remain regarding how to apply this geometric insight to exactly compute, or even simply bound, the nonnegative integer rank beyond the rank~2 case. 

\paragraph{Acknowledgments} \

The first author was partially funded by the Centre for Mathematics of the University of Coimbra (CMUC, https://doi.org/10.54499/UID/00324/2025) under the Portuguese Foundation for Science and Technology (FCT), through Grants UID/00324/2025, and UID/PRR/00324/2025. The second author was partially funded by the Natural Sciences and Engineering Research Council of Canada (cette recherche est partiellement financ\'ee par le Conseil de recherches en sciences naturelles et en g\'enie du Canada), Discovery  Grant \#2024-04643.

\begin{algorithm}
\caption{NNI Rank 2 Factorization Algorithm} \label{alg:NNI2}
    \SetKwRepeat{Do}{do}{while}
    \DontPrintSemicolon
        \KwIn{Integer matrix $A = [a_1\cdots a_n]\in\mathbb{Z}^{m\times n}$; index of canonization $r$.}

        Initialize factors $F_1 = [\,], F_2 = [\,]$
        
        Compute basis $B$ for $\col(A)\cap \mathbb{Z}^m$

        Compute generators $c_1,c_2$ of $\col(A)\cap\RR_+^m$ 

        \For{$i$ from $1$ to $m$}
        {
        Solve $Bx_i = a_i$ for $x_i\in\mathbb{Z}^2$
        }
        Find canonical transform $T$ such that $Tc_r = e_1$
        
        \For{$v \in [c_1,c_2,x_1,\ldots, x_m]$}
        {
        $v \gets Tv$ \;
        }  
        Decompose $K_A = K_{-}\cup K\cup K_+$

        Set $u$ to be the generator of $K_{-}$ that is not $e_1$
        
        \For{$ka \in K_{-}\cap(u-K_+)\cap\ZZ^2_+$}
        {
        Set $a\in\ZZ^2$ as primitive generator of $\RR_+(ka)$
        
        Set $b\in\ZZ^2$ as primitive generator of $u-ka$

        \If{$b\neq 0$}
        {
        \For{$i$ from $1$ to $m$}
        {
        $w_i = (a,b)^{-1}x_i$
        
        \If{$w_i\notin\ZZ_+^2$}
        {
        next $ka$
        }
        }
        $F_2 \gets [w_1\cdots w_m] \in \ZZ_+^{2\times m}$\;
        $F_1 \gets B [a\; b]$ \;
        
        \Return{$F_1,F_2$}
        }
        \Else
        {
        $v \gets K\cap K_+$

        $k\gets 0$
        
        \While{$v-ka \in K_A$}
        {
        Set $b$ as primitive ray generator of $v-ka$

        \For{$i$ from $1$ to $m$}
        {
        $w_i = (a,b)^{-1}x_i$
        
        \If{$w_i\notin\ZZ_+^2$}
        {
        $k \gets k+1$
        
        \If{$v-ka \notin K_A$}
        {
        next $ka$
        }
        next $b$
        }
        }
        
        $F_2 \gets [w_1\cdots w_m] \in \ZZ_+^{2\times m}$\;
        $F_1 \gets B [a\; b]$ \;
        \Return{$F_1,F_2$}
        }
        }
        }

    \end{algorithm}

\bibliography{stuff}{}
\bibliographystyle{alpha}

\end{document}